\documentclass[10pt]{amsart}

\usepackage{amsmath, amscd, amssymb,xypic, amsfonts, epsfig}
\usepackage[frame,cmtip,arrow,matrix,line,graph,curve]{xy}
\usepackage{graphpap, color, pstricks}
\usepackage[mathscr]{eucal}

\usepackage{mathrsfs}
\usepackage{pstricks}
\usepackage{color}
\usepackage{cancel}
\usepackage{verbatim}
\usepackage{pstricks}
\usepackage{color}
\usepackage{cancel}

\numberwithin{equation}{section}

\newcommand{\CC}{\mathbb{C}}

\newcommand{\LL}{\mathbb{L}}
\newcommand{\PP}{\mathbb{P}}
\newcommand{\QQ}{\mathbb{Q}}

\newcommand{\ZZ}{\mathbb{Z}}

\newcommand{\TT}{\mathbb{T}}

\newcommand{\cal}{\mathcal}

\def\cD{{\cal D}}

\def\cH{{\cal H}}

\def\cL{{\cal L}}

\def\cO{{\cal O}}

\def\cU{{\cal U}}

\def\mapright#1{\,\smash{\mathop{\lra}\limits^{#1}}\,}

\def\virt{^{\mathrm{vir}}}

\def\lra{\longrightarrow}

\def\lsta{_{\ast}}

\def\begeq{\begin{equation}}
\def\endeq{\end{equation}}
\def\and{\quad{\rm and}\quad}

\def\sub{\subset}

\def\Po{{\mathbb P^1}}
\def\and{\quad\text{and}\quad}
\def\mapright#1{\,\smash{\mathop{\lra}\limits^{#1}}\,}

\newtheorem{prop}{Proposition}[section]
\newtheorem{theo}[prop]{Theorem}
\newtheorem{lemm}[prop]{Lemma}
\newtheorem{coro}[prop]{Corollary}
\newtheorem{rema}[prop]{Remark}
\newtheorem{exam}[prop]{Example}
\newtheorem{ques}[prop]{Question}
\newtheorem{defi}[prop]{Definition}

\def\beq{\begin{equation}}
\def\eeq{\end{equation}}

\def\loc{_{\mathrm{loc}}}

\def\PP{\mathbb{P}}
\def\CC{\mathbb{C}}

\def\lra{\longrightarrow}
\def\mapright#1{\,\smash{\mathop{\lra}\limits^{#1}}\,}
\def\cO{\mathcal{O}}

\def\cX{\mathcal{X} }

\def\bbA{\mathbb{A} }

\title{Poincar\'e invariants are Seiberg-Witten invariants}
\author{Huai-liang Chang}
\address{Department of Mathematics, Hong Kong University of Science and Technology, Hong Kong} \email{mahlchang@ust.hk}
\author{Young-Hoon Kiem}
\address{Department of Mathematics and Research Institute
of Mathematics, Seoul National University, Seoul 151-747, Korea}
\email{kiem@math.snu.ac.kr}
\thanks{HLC was partially supported by Hong Kong GRF grant 600711. YHK was partially supported by Korea NRF grant 2011-0027969}

\begin{document}
\begin{abstract}
We prove a conjecture of D\"urr, Kabanov and Okonek which provides an algebro-geometric theory of Seiberg-Witten invariants for all smooth projective surfaces. Our main technique is the cosection localization principle (\cite{KLi-cosection}) of virtual cycles. 
\end{abstract}

\maketitle

\def\loc{_{\mathrm{loc}} }
\section{Introduction}

Recently there has been a renewed interest in Donaldson invariants and Seiberg-Witten invariants due to the influx of virtual intersection theory. See \cite{Moch, GNY, DKO} for instance. The purpose of this paper is to prove a conjecture (Theorem \ref{thref1} below) of D\"urr, Kabanov and Okonek in \cite{DKO} which provides a natural algebro-geometric theory of Seiberg-Witten invariants. Our main technique is the cosection localization principle in \cite{KLi-cosection} that effectively localizes the virtual cycle when there is a cosection of the obstruction sheaf. 

In mid 1980s, Donaldson defined his famous invariants as intersection numbers on the Uhlenbeck compactification of the space of anti-self-dual (ASD) connections on a fixed hermitian vector bundle of rank 2 on a compact oriented 4-manifold  $X$ (\cite{Dona}). Because of difficulty in calculating Donaldson invariants, an algebro-geometric theory of Donaldson invariants was highly anticipated from the beginning. Donaldson proved that when the 4-manifold $X$ is an algebraic surface over $\CC$, there is a diffeomorphism between the space of irreducible ASD connections and an open subset of the moduli space of Gieseker semistable sheaves of rank 2 and given Chern classes. In 1991, J. Li (\cite{JLi}) and Morgan (\cite{Morg}) extended Donaldson's diffeomorphism to a continuous map from the Gieseker moduli space of semistable sheaves to the Uhlenbeck compactification  and proved that Dondalson invariants are intersection numbers on the Gieseker moduli space.  In fact, J. Li furthermore proved that the Uhlenbeck compactification admits a scheme structure and the map from the Gieseker moduli space to the Uhlenbeck compactificaiton is an algebraic morphism. In 1993, Kronheimer and Mrowka proved the celebrated structure theorem which expresses all the Donaldson invariants in terms of a finite number of classes $K_1,\cdots, K_l \in H^2(X,\ZZ)$ and rational numbers $\alpha_1,\cdots,\alpha_l\in \QQ$ if $X$ is of simple type (\cite{KrMr}). The condition of being a simple type roughly means that the point insertions do not provide new information on $X$. The mystery of the simple type condition, the basic classes $K_1,\cdots, K_l$ and the rational numbers $\alpha_1,\cdots,\alpha_l$ was elucidated by the advent of Seiberg-Witten theory in 1994.

A $Spin^c$-structure on a 4-manifold $X$ refers to a pair of rank 2 hermitian vector bundles $E^\pm$ such that $\det E_+\cong \det E_-=:L$. Taking the first Chern class of $L$ provides us with a bijection from the collection of all $Spin^c$-structures on $X$ to $H^2(X,\ZZ)$.
Seiberg and Witten stated a pair of equations on a pair $(A, \varphi)$ where $A$ is a connection on $L$  and $\varphi$ is a section of $E_+$. The collection of all solutions of Seiberg-Witten equations forms a compact topological space and Seiberg-Witten invariants are defined as intersection numbers on the solution space. 
In 1994, Witten in \cite{Witt} conjectured that every K\"ahler surface $X$ with a nontrivial holomorphic 2-form $\theta\in H^0(K_X)$ is of simple type and that  for any K\"ahler surface $X$ of simple type, \begin{enumerate}\item the basic classes $K_1,\cdots, K_l$ of Kronheimer and Mrowka satisfy $$K_i\cdot(K_i -k_X)=0\quad\forall i\quad\text{where }k_X=c_1(T^*_X)$$
\item the Seiberg-Witten invariants $SW(\gamma)$ are zero if $\gamma\cdot(\gamma-k_X)\ne 0$;
\item the rational numbers $\alpha_i$ in the structure theorem of Kronheimer and Mrowka are the Seiberg-Witten invariants $SW(K_i)$ upto a constant which depends only on $b_1(X),b_2^\pm(X)$.
\end{enumerate}
Furthermore, Witten showed by physical means that the calculation of Seiberg-Witten invariants may be localized to a neighborhood of a canonical divisor when $p_g(X)>0$. (See \cite[p.12]{Witt}, \cite[p.54]{Dona-Bull}) 

When $X$ is a K\"ahler surface with $b_1(X)=0$, it was observed by Witten (\cite[p.18]{Witt}) that the solution space of Seiberg-Witten equations with fixed $\gamma=c_1(L)\in H^2(X,\ZZ)$ is a projective space $\PP H^0(X,L)$ and a theorem of Friedman and Morgan \cite[Theorem 3.1]{FM} shows that the Seiberg-Witten invariants in this case are the integrals of cohomology classes multiplied by the Euler class of a certain vector bundle. Hence in the special case of $b_1(X)=0$, we have an algebro-geometric theory of Seiberg-Witten invariants. Using this, T. Mochizuki in \cite{Moch} proved a formula that expresses the Donaldson invariants in terms of the Seiberg-Witten invariants of surfaces with $b_1(X)=0$. Subsequently in \cite{GNY}, G\"ottsche, Nakajima and Yoshioka proved that Mochizuki's formula implies Witten's conjecture for algebraic surfaces with $b_1(X)=0$. However this beautiful story could not be generalized to the case where $b_1(X)>0$ because we still lack in an algebro-geometric definition of Seiberg-Witten invariants. Moreover, the proofs of Mochizuki and G\"ottsche-Nakajima-Yoshioka do not seem to explain the localization behavior of Seiberg-Witten invariants to a canonical divisor.

In 2007, D\"urr, Kabanov and Okonek proved in \cite{DKO} that if $X$ is a smooth projective surface and $\gamma\in H_2(X,\ZZ)$, the Hilbert scheme $Hilb^\gamma_X$ of divisors $D$ on $X$ whose homology classes are $\gamma$ admits a perfect obstruction theory and thus we obtain a virtual fundamental class $[Hilb^\gamma_X]\virt$ by \cite{LT, BF}. By integrating cohomology classes over $[Hilb^\gamma_X]\virt$, they defined new invariants of $X$ called the \emph{Poincar\'e invariants} and conjectured that the Poincar\'e invariants coincide with the Seiberg-Witten invariants for algebraic surfaces. See \S3 for more details. Our main result in this paper is that the following conjecture of D\"urr, Kabanov and Okonek in \cite{DKO} is true. 

\begin{theo}\label{thref1} The Poincar\'e invariants are the Seiberg-Witten invariants for all smooth projective surfaces. \end{theo}

This theorem gives us a completely algebro-geometric definition of Seiberg-Witten invariants for all smooth projective surfaces and can be thought of as a natural generalization of \cite[Theorem 3.1]{FM}. Since $Hilb^\gamma_X$ parameterizes embedded curves, the Poincar\'e invariants may be viewed as ``algebro-geometric Gromov invariants" and Theorem \ref{thref1} may be considered as an algebraic version of Taubes's Theorem (\cite{Taubes}).

The authors of \cite{DKO} proved deformation invariance, blow-up formula and wall crossing formulas for the Poincar\'e invariants and reduced the proof of Theorem \ref{thref1} to the following (\cite[p.286]{DKO}).

\begin{theo}\label{prref1} Let $X$ be a minimal surface of general type.
If $p_g(X) >0$, $$\deg [Hilb^{k_X}_X]\virt = (-1)^{\chi(\cO_X)}$$ where $k_X$ is the homology class of a canonical divisor.
\end{theo}

The main technique for our proof of Theorem \ref{thref1} and Theorem \ref{prref1} is the cosection localization principle (\cite{KLi-cosection}) which tells us that if there is a cosection
\[\sigma: \cO b_M\lra \cO_M \]
of the obstruction sheaf $\cO b_M=h^1(E^\vee)$ of a perfect obstruction theory $\phi:E\to \mathbb{L}_M$ over a Deligne-Mumford stack $M$, then the virtual fundamental class of $(M,\phi)$ localizes to the zero locus of $\sigma$. Suppose $p_g(X)>0$ and let $\theta\in H^0(X,K_X)$ be a nonzero holomorphic 2-form on $X$ whose vanishing locus is denoted by $Z$. For $D\in Hilb^\gamma_X$, the obstruction space at $D$ with respect to the perfect obstruction theory in \cite{DKO} is $H^1(\cO_D(D))$. From the short exact sequence $$0\lra \cO_X\lra \cO_X(D)\lra \cO_D(D)\lra 0$$
we obtain a connecting homomorphism $H^1(\cO_D(D))\to H^2(\cO_X)$. Upon composing with the multiplication by $\theta$, we obtain a homomorphism
\[ Ob_{Hilb^\gamma_X,D}=H^1(\cO_D(D))\lra H^2(\cO_X)\mapright{\theta} H^2(\cO_X(Z))\cong H^2(K_X)=\CC.  \]
Obviously this construction can be lifted to a cosection \[\sigma:\cO b_{Hilb^\gamma_X}\cong R^1\pi_*\cO_{\cD}(\cD)\lra \cO_{Hilb^\gamma_X}\]
where 
\[\xymatrix{
\cD\ar@{^(->}[r]\ar[dr] & Hilb^\gamma_X\times X\ar[d]^\pi\\ & Hilb^\gamma_X}\] is the universal family of divisors on $X$. The vanishing locus of the cosection $\sigma$ is precisely the locus of $D\in Hilb^\gamma_X$ satisfying $0\le D\le Z$. Therefore the virtual fundamental class is localized to the locus of effective divisors contained in $Z$ and  the calculation of the Poincar\'e invariants takes place within the canonical divisor $Z$, exactly as Witten told us about localization of Seiberg-Witten invariants mentioned above.

When $\gamma=k_X$, we will see that the vanishing locus of $\sigma$ consists of exactly one reduced point $Z$. Hence the virtual cycle of $Hilb^{k_X}_X$ is localized to a neighborhood of the point $Z$. 
By using the results of Green and Lazarsfeld (\cite{GL1, GL2}) on deforming cohomology groups of line bundles, we will find local defining equations near canonical divisors and show that there is a canonical divisor $Z$ which is a smooth point of $Hilb^{k_X}_X$ such that $\dim T_ZHilb^{k_X}_X$ has the same parity as $\chi(\cO_X)$. By \cite[Example 2.9]{KLi-cosection}, this implies that the (localized) virtual cycle of $Hilb^{k_X}_X$ is $(-1)^{\chi(\cO_X)}[Z]$ whose degree is precisely $(-1)^{\chi(\cO_X)}$. This proves Theorem \ref{prref1}. 

By Theorem \ref{thref1}, Mochizuki's formula in \cite[Chapter 7]{Moch} expresses the Donaldson invariants in terms of the Seiberg-Witten invariants. Therefore one may be able to generalize the arguments of \cite{GNY} to answer the following interesting question.
\begin{ques} Does Mochizuki's formula imply Witten's conjecture for all smooth projective surfaces $X$ with $p_g(X)>0$?
\end{ques}
We hope to get back to this question in the future.

\bigskip

\noindent \textbf{Acknowledgement}. We are grateful to Jun Li for many useful discussions.

\bigskip

\def\cX{\mathcal{X} }
\def\bbA{\mathbb{A} }
\section{Localization of virtual cycles by cosections}\label{sec2}
In this section we collect necessary materials on the cosection localization principle from \cite{KLi-cosection}. 
\begin{defi} Let $M$ be a Deligne-Mumford stack over $\CC$. Let $\LL_M$ denote the cotangent complex of $M$.
A \emph{perfect obstruction theory} on $M$ is a morphism $\phi:E\to \LL_M$ in the derived category $D^b(M)$ of bounded complex of coherent sheaves on $M$ such that 
\begin{enumerate}\item $E$ is locally isomorphic to a two-term complex of locally free sheaves concentrated at $[-1,0]$;
\item $h^{-1}(\phi)$ is surjective and $h^0(\phi)$ is an isomorphism.\end{enumerate}
The \emph{obstruction sheaf} of $(M,\phi)$ is defined as $\cO b_M=h^1(E^\vee)$ where $E^\vee$ denotes the dual of $E$. A \emph{cosection} of the obstruction sheaf $\cO b_M$ is a homomorphism $\cO b_M\to \cO_M$.
\end{defi}
By the construction in \cite{BF, LT}, a perfect obstruction theory $\phi$ on $M$ gives rise to a \emph{virtual fundamental class} $[M]\virt$ and many well-known invariants (such as Gromov-Witten and Donaldson-Thomas invariants) are defined as intersection numbers on the virtual fundamental classes of suitable moduli spaces. The cosection localization principle of \cite{KLi-cosection} is
 a powerful technique of calculating these virtual intersection numbers.
\begin{theo} (\cite[Theorem 1.1]{KLi-cosection})\label{thref2} Suppose there is a \emph{surjective} cosection $\sigma:\cO b_M|_U\to \cO_U$ over an open $U\subset M$. Let $M(\sigma)=M-U$. Then the virtual fundamental class localizes to $M(\sigma)$ in the sense that there exists a \emph{localized virtual fundamental class}
\[ [M]\virt\loc\in A_*(M(\sigma))\]
which enjoys the usual properties of the virtual fundamental classes and such that 
\[ \imath_*[M]\virt\loc=[M]\virt\in A_*(M)\quad\text{where }\imath:M(\sigma)\hookrightarrow M.\]
\end{theo}
See \cite{KLi-cosection, KLi-surf1, KLi-surf2} for direct applications of Theorem \ref{thref2} to Gromov-Witten invariants of surfaces. From the construction of $[M]\virt\loc$ in \cite{KLi-cosection}, the following excision property follows immediately. 
\begin{prop}
Let $W$ be an open neighborhood of $M(\sigma)$ in $M$. Then we have $$[W]\virt\loc=[M]\virt\loc\in A_*(M(\sigma)).$$
\end{prop}

The following special case will be useful.
\begin{exam}(\cite[Example 2.9]{KLi-cosection})\label{exref1}
Let $M$ be an $n$-dimensional smooth scheme and $E$ be a vector bundle of rank $n$ on $M$. 
The zero map $0:\TT_M\to E$ is a perfect obstruction theory with obstruction sheaf $E$. Let $\sigma:E\to \cO_M$ be a cosection whose vanishing locus $\sigma^{-1}(0)$ is a simple point $p$ in $M$. Then  $[M]\virt\loc=(-1)^n[p]$.
\end{exam}


\section{Poincar\'e invariants}\label{sec3}
In this section, we recall the definition of Poincar\'e invariants from \cite{DKO} as virtual intersection numbers on the Hilbert scheme $Hilb^\gamma_X$ of divisors on a smooth projective surface $X$ with $p_g(X)>0$. For any nonzero $\theta\in H^0(X,K_X)$, we construct a cosection $\sigma_\theta: \cO b_{Hilb^\gamma_X}\to \cO_{Hilb^\gamma_X}$ of the obstruction sheaf and show that the vanishing locus of $\sigma$ is the locus of divisors $D$ contained in the zero locus $Z$ of $\theta$.

\subsection{Perfect obstruction theory on Hilbert scheme}\label{sec3.1}
In this subsection, we recall the perfect obstruction theory on the Hilbert scheme $Hilb^\gamma_X$ of divisors on $X$ and the Poincar\'e invariants from \cite{DKO}. 

By combining Theorem 1.7 and Theorem 1.11 of \cite{DKO}, we obtain the following.
\begin{theo}\label{thref5} Let $\cX\to S$ be a flat quasi-projective morphism of relative dimension $2$ and $\gamma\in H_2(\cX,\ZZ)$. Let $Hilb^\gamma_{\cX/S}$ be (the open subscheme of) the relative Hilbert scheme parameterizing Cartier divisors $D$ of fibers of $\cX\to S$ with $[D]=\gamma\in H_2(\cX,\ZZ)$. Let 
\[\xymatrix{ \cD\ar@{^(->}[r]\ar[dr] & Hilb^\gamma_{\cX/S}\times_S\cX\ar[r]\ar[d]^\pi & \cX\ar[d]\\
& Hilb^\gamma_{\cX/S} \ar[r] & S}\]
be the universal family. 
Then we have a relative perfect obstruction theory
\[
\phi:(R\pi_*\cO_\cD(\cD))^\vee\lra \LL_{Hilb^\gamma_{\cX/S}/S}
\]
for the morphism $Hilb^\gamma_{\cX/S}\to S$.
\end{theo}
When $S=\mathrm{Spec}\, \CC$ and $\imath:\cD\hookrightarrow Hilb^\gamma_{X}\times X$ is a Cartier divisor (e.g. when $X$ is smooth), we have a perfect obstruction theory on $Hilb^\gamma_X$ whose obstruction sheaf is
\[ \cO b_{Hilb^\gamma_X} = R^1\pi_*\cO_\cD(\cD) \]
and by \cite{BF,LT} we obtain a virtual fundamental class $[Hilb^\gamma_X]\virt$ if $X$ is projective. We will see below if there is a nonzero section $\theta\in H^0(K_X)$ for smooth $X$, there is a cosection 
\[ \sigma_\theta:\cO b_{Hilb^\gamma_X}\lra \cO_{Hilb^\gamma_X}\]
which enables us to define the localized virtual fundamental class $[Hilb^\gamma_X]\virt\loc$ supported on the zero locus $Z$ of $\theta$.

Let $X$ be a smooth projective surface. Then it is easy to find that the virtual dimension of $Hilb^\gamma_X$ is precisely $\gamma\cdot(\gamma-k_X)$ where $k_X=c_1(K_X)$. 
The Poincar\'e invariants for $X$ are now defined as intersection numbers on $[Hilb^\gamma_X]\virt$ but the precise definition is not necessary in this paper. See \cite[\S0]{DKO} for the precise definition.

It was conjectured in \cite{DKO} that the Poincar\'e invariants for $X$ coincide with the Seiberg-Witten invariants. Furthermore, the authors of \cite[p.286]{DKO} proved that the conjecture follows if 
$$ \deg [Hilb^{k_X}_X]\virt=(-1)^{\chi(\cO_X)} $$
for minimal surfaces of general type with $p_g>0$ (Theorem \ref{prref1}).

\subsection{Cosection of the obstruction sheaf}\label{sec3.2}

Suppose $p_g(X)>0$ and fix a nonzero holomorphic 2-form $\theta\in H^0(X,K_X)$ on $X$ whose vanishing locus is denoted by $Z$. For $D\in Hilb^\gamma_X$, the obstruction space at $D$ with respect to the perfect obstruction theory in \cite{DKO} is $H^1(\cO_D(D))$. From the short exact sequence $$0\lra \cO_X\lra \cO_X(D)\lra \cO_D(D)\lra 0$$
we obtain a connecting homomorphism $H^1(\cO_D(D))\to H^2(\cO_X)$. Upon composing with the multiplication by $\theta$, we obtain a homomorphism
\begin{equation}\label{eqref2}\sigma_\theta: Ob_{Hilb^\gamma_X,D}=H^1(\cO_D(D))\lra H^2(\cO_X)\mapright{\theta} H^2(\cO_X(Z))= H^2(K_X)=\CC.  \end{equation}
 
This construction can be lifted to a cosection $\sigma:\cO b_{Hilb^\gamma_X}\to \cO_{Hilb^\gamma_X}$ as follows. Let
\[\xymatrix{
\cD\ar@{^(->}[r]\ar[dr] & Hilb^\gamma_X\times X\ar[d]_{\pi}\\ & Hilb^\gamma_X}\] be the universal family and let \begin{equation}\label{eqref1} 0\lra \cO_{Hilb^\gamma_X\times X}\lra \cO_{Hilb^\gamma_X\times X}(\cD)\lra \cO_{\cD}(\cD)\lra 0\end{equation} be the obvious short exact sequence.
Recall from \cite[\S1]{DKO} that our perfect obstruction theory of $Hilb^\gamma_X$ is
$$\phi:(R\pi_*\cO_{\cD}(\cD))^\vee\lra \mathbb{L}_{Hilb^\gamma_X}$$ and thus the obstruction sheaf of $(Hilb^\gamma_X,\phi)$ is $$\cO b_{Hilb^\gamma_X}= R^1\pi_*\cO_{\cD}(\cD).$$
From \eqref{eqref1}, we obtain a homomorphism $$R^1\pi_*\cO_{\cD}(\cD)\lra R^2\pi_*\cO_{Hilb^\gamma_X\times X}$$ and by composing it with the multiplication by the pullback $p_X^*\theta$ of $\theta$ via the projection $p_X:Hilb^\gamma_X\times X\to X$ we obtain the cosection 
\begin{equation}\label{eqref5}\sigma_\theta:\cO b_{Hilb^\gamma_X}\cong R^1\pi_*\cO_{\cD}(\cD)\lra R^2\pi_*\cO_{Hilb^\gamma_X\times X}\mapright{p_X^*\theta} R^2\pi_*p_X^*K_X=\cO_{Hilb^\gamma_X}\end{equation}
\begin{lemm}\label{lem-DegIsLoc} The vanishing locus of the cosection $\sigma_\theta$ is precisely the locus of $D\in Hilb^\gamma_X$ satisfying $0\le D\le Z$. \end{lemm}
\begin{proof} Let $s\in H^0(\cO_X(D))$ be a nonzero section whose vanishing locus is $D$.
We have to show that \eqref{eqref2} vanishes if and only if $D\le Z$. From the short exact sequence \[ 0\lra \cO_X\lra \cO_X(D)\lra \cO_D(D)\lra 0,\]
\eqref{eqref2} vanishes  if and only if the second homomorphism in \eqref{eqref2} factors as \[ \theta: H^2(\cO_X)\mapright{s} H^2(\cO_X(D))\lra H^2(\cO_X(Z))=H^2(K_X)\]
By taking the duals, we find that this is equivalent to saying that $\theta:H^0(\cO_X)\to H^0(K_X)$ factors as
\[ \theta:H^0(\cO_X)\lra H^0(K_X-D)\mapright{s} H^0(K_X).\]
Let $s'$ be the image of the first homomorphism and $D'$ be the zero locus of $s'$. Then we have $\theta=ss'$ and $Z=D+D'$ with $D'\ge 0$. This certainly implies that \eqref{eqref2} vanishes if and only if $0\le D\le Z$. 
\end{proof}

By the cosection localization principle in \cite{KLi-cosection} (see \S\ref{sec2}), we obtain the following.
\begin{prop} Let $Hilb^\gamma_X(Z)$ be the locus of $D\in Hilb^\gamma_X$ with $0\le D\le Z$ and let $\imath:Hilb^\gamma_X(Z)\hookrightarrow Hilb^\gamma_X$ denote the inclusion morphism.
There exists a localized virtual fundamental class $[Hilb^\gamma_X]\virt_{\mathrm{loc}}\in A_*(Hilb^\gamma_X(Z))$ such that $\imath_*[Hilb^\gamma_X]\virt_{\mathrm{loc}}\in A_*(Hilb^\gamma_X)$ is the ordinary virtual fundamental class $[Hilb^\gamma_X]\virt$ in \S\ref{sec3.1}. Furthermore, if $W$ is an open neighborhood of $Z$ in $Hilb^\gamma_X$, we have $[W]\virt\loc=[Hilb^\gamma_X]\virt_{\mathrm{loc}}$.
\end{prop}
Therefore the calculation of the Poincar\'e invariants takes place near a canonical divisor $Z$. This is consistent with Witten's claim about localization of Seiberg-Witten invariants to a canonical divisor (\cite{Witt,Dona-Bull}).

Suppose $\gamma=k_X:=c_1(K_X)\in H^2(X,\ZZ)$. Then it is obvious that $Hilb^\gamma_X(Z)=\{Z\}$ as a set. Moreover, the virtual dimension $\gamma\cdot(\gamma-k_X)$ of $Hilb^\gamma_X$ is $0$.
\begin{coro}\label{cor-locdeg}
When $\gamma=k_X$, the localized virtual fundamental class $[Hilb^\gamma_X]\virt_{\mathrm{loc}}$ of $Hilb^\gamma_X$ is supported at a single point $\{Z\}$ and the Poincar\'e invariant of $Hilb^{k_X}_X$ is the degree of $[Hilb^\gamma_X]\virt_{\mathrm{loc}}$. \end{coro}

\subsection{Another construction of the cosection}\label{sec3.25} There is another way to define the cosection $$\sigma=\sigma_\theta:\cO b_{Hilb^\gamma_X}\lra \cO_{Hilb^\gamma_X}.$$ Let $\theta\in H^0(X,K_X)$ be a nonzero section and $D\in Hilb^{\gamma}_X$ be a divisor with $c_1(\cO_X(D))=\gamma$. Since $D$ is a local complete intersection with normal bundle $\cO_D(D)$, 
$D$ is a Cohen-Macaulay scheme (\cite[II, Proposition 8.23]{Hart}) and the adjunction formula tells us that the dualizing sheaf of $D$ is given by
\begin{equation}\label{eqref6}\omega_D=K_X\otimes \cO_D(D).\end{equation}
Multiplication by $\theta$ gives us a map 
\[ \sigma|_D:H^1(\cO_D(D))\mapright{\theta} H^1(K_X\otimes \cO_D(D))\cong H^1(\omega_D) \cong H^0(\cO_D)^\vee\]
from the obstruction space $H^1(\cO_D(D))$ at $D\in Hilb^{\gamma}_X$. As in \S\ref{sec3.2}, this construction can be lifted to the whole space as
\begin{equation}\label{eqref4}\sigma:R^1\pi_*\cO_\cD(\cD)\lra R^1\pi_*(K_X\otimes \cO_\cD(\cD))= R^1\pi_*(\omega_{\cD/Hilb^\gamma_X})\cong \pi_*\cO_{\cD}^\vee.\end{equation}

\begin{lemm}\label{leref1} Suppose $H^1(K_X(D))=0$ for all $D\in Hilb^\gamma_X$ and $\gamma\ne 0$. Then $\pi_*\cO_{\cD}\cong \cO_{Hilb^\gamma_X}$ and \eqref{eqref4} gives us a cosection which coincides with \eqref{eqref5}.
\end{lemm}
\begin{proof} 
From the short exact sequence,
\[ 0\lra \cO_{Hilb^\gamma_X\times X}(-\cD)\lra \cO_{Hilb^\gamma_X\times X}\lra \cO_{\cD}\lra 0,\]
we have an exact sequence 
\[ 0\lra \pi_*\cO_{Hilb^\gamma_X\times X}(-\cD)\lra \pi_*\cO_{Hilb^\gamma_X\times X}\lra \pi_*\cO_\cD\] \[ \lra R^1\pi_*\cO_{Hilb^\gamma_X\times X}(-\cD)\cong (R^1\pi_*(K_X\otimes \cO_{Hilb^\gamma_X\times X}(\cD)))^\vee.\]
Since any $D\in Hilb^\gamma_X$ is effective and nonzero, $\pi_*\cO_{Hilb^\gamma_X\times X}(-\cD)=0$ and $R^2\pi_*(K_X\otimes\cO_{Hilb^\gamma_X\times X}(\cD))=0$ by Serre duality. By assumption, $R^1\pi_*K_X\otimes\cO_{Hilb^\gamma_X\times X}(\cD)=0$ and we thus have the isomorphisms $\pi_*\cO_\cD\cong \pi_*\cO_{Hilb^\gamma_X\times X}\cong\cO_{Hilb^\gamma_X}.$
 
Next, consider the commutative diagram of exact sequences
\[\xymatrix{
0 \ar[r] & \cO_X\ar[r]\ar[d]_\theta & \cO_X(D)\ar[r]\ar[d]_\theta & \cO_D(D)\ar[r]\ar[d]_\theta & 0 \\
0\ar[r] & K_X\ar[r] & K_X(D)\ar[r] & \omega_D\ar[r] &0
}\]
where the vertical arrows are all multiplication by $\theta$. Here we used \eqref{eqref6}. By taking cohomology, we obtain a commutative diagram of exact sequences
\[\xymatrix{
&H^1(\cO_D(D)) \ar[r]^\delta\ar[d]_\theta & H^2(\cO_X)\ar[d]_\theta \\
H^1(K_X(D))\ar[r] & H^1(\omega_D) \ar[r] & H^2(K_X)\ar[r] & H^2(K_X(D)).
}\]
Since $D$ is effective and nonzero, $H^2(K_X(D))\cong H^0(\cO_X(-D))^\vee=0$. 
Hence the middle bottom arrow is surjective which is also injective because $H^1(K_X(D))=0$. The right vertical composed with the top horizontal is \eqref{eqref5} and the left vertical is \eqref{eqref4}. It is obvious that this comparison lifts to the whole family $\cD\subset Hilb^\gamma_X\times X$. 
\end{proof}

\begin{rema}\label{rema1} Let $X$ be a smooth projective surface which is minimal of general type. By \cite[VII, \S2]{BPV}, $K_X$ is big and nef. Since nefness and bigness are obviously preserved by numerical equivalence, $\cO_X(D)$ is big and nef for any $D\in Hilb^{k_X}_X$. By \cite[I, Theorem 4.3.1]{Larz}, we have the vanishing
\[ H^i(X,K_X(D))=0,\quad  i=1,2\]
for any $D\in Hilb^{k_X}_X$. \end{rema}

By Lemma \ref{leref1} and Remark \ref{rema1}, we obtain
\begin{coro}\label{coref1} Let $X$ be a smooth projective surface which is minimal and of general type. The two definitions of cosections by \eqref{eqref5} and \eqref{eqref4}  coincide for $Hilb^{k_X}_X$ if $X$ is a minimal surface of general type.\end{coro}

The scheme structure of the zero (or degeneracy) locus of $\sigma$ is actually reduced.
\begin{lemm} \label{DegIsRed} Let $X$ be a smooth projective surface which is minimal and of general type. The zero locus of the cosection $\sigma$ is the \emph{simple} point $Z=\mathrm{zero}(\theta)$.\end{lemm}
\begin{proof}
By Lemma \ref{lem-DegIsLoc}, we only need to show that $\mathrm{zero}(\sigma)$ is the reduced point $\{Z\}$. 
The dual of  \eqref{eqref4} 
is obtained by multiplication $$\cO_{Hilb^{k_X}_X}=\pi_*\cO_\cD\mapright{\theta} \pi_*(K_X|_\cD)$$
by $\theta$. Hence for any scheme $T$, a morphism $T\to Hilb^{k_X}_X$ factors through $\mathrm{zero}(\sigma)$ if and only if we have a flat family of divisors $\cD_T\to T$ that fits into a diagram
\[\xymatrix{
\cD_T\ar@{^(->}[r]\ar[dr] & Z\times T\ar@{^(->}[r]\ar[d] &X\times T\ar[dl] \\
& T.
}\]
But since $D\in Hilb^{k_X}_X$ and $Z$ have the same Hilbert polynomial, $\cD_T=Z\times T$ and hence the morphism $T\to \mathrm{zero}(\sigma)$ factors through the reduced point $\{Z\}$.
This proves that $\mathrm{zero}(\sigma)$ is indeed the reduced point $\{Z\}$.
\end{proof}

\def\Po{\mathbb{P}^1 }

\section{A proof of Theorem \ref{thref1}} \label{sec4}
In this section we will prove Theorem \ref{prref1} and thus complete a proof of Theorem \ref{thref1}. Let $X$ be a minimal projective surface of general type with $p_g>0$. We will find local defining equations of $Hilb^{k_X}_X$ near canonical divisors and show that there is a canonical divisor $Z$ representing a smooth point of the Hilbert scheme whose dimension at $Z$ has the same parity as $\chi(\cO_X)$. Then Theorem \ref{prref1} will follow directly from Example \ref{exref1}. 

Let $M=Hilb^{k_X}_X$ denote the Hilbert scheme of divisors $D$ with $c_1(\cO_X(D))=k_X=c_1(K_X)\in H^2(X,\ZZ)$. Let $P=Pic^{k_X}(X)$ (resp. $P^{0}=Pic^{0}(X)$) denote the Picard variety of line bundles $L$ on $X$ with $c_1(L)=k_X$ (resp. $c_1(L)=0$). Since a Cartier divisor defines a line bundle, we have a natural morphism
$$\tau: M \to P$$
whose fiber over an $L\in P$ is the complete linear system $\PP H^0(X, L)$.
It is easy to see that $M$ is the fine moduli space of pairs $(L,s)$ where $L\in P$ and $s\in H^0(X,L)-\{0\}$ where two such pairs $(L,s)$ and $(L',s')$ are isomorphic if there is an isomorphism $L\mapright{\cong} L'$ that sends $s$ to $s'$. 

Let $\phi_1,\phi_2,\cdots,\phi_q$ be a basis of $H^1(X,\cO_X)$ and let $T$ be a ball in $H^1(\cO_X)\cong\CC^q=\{(t_1,\cdots,t_q)\}$. Consider the isomorphism $P^0\to P$ defined by $L\mapsto L^{-1}\otimes K_X$ and the exponential map $H^1(\cO_X)\to P^0$ that sends $0$ to the trivial line bundle $\cO_X$.
Pulling back the universal family over $P\times X$ by the composition $H^1(\cO_X)\to P^0\to P$, we obtain a family $\cL\to T\times X$ of line bundles such that $\cL|_{0\times X}\cong K_X$. Let $\pi$, $\rho$ denote the projections from $T\times X$ to $T$ and $X$ respectively. 

Let $D^\bullet$ denote the complex
\beq \label{eq-GL2} 0\lra H^0(\cO_X)\otimes \cO_T \lra H^1(\cO_X)\otimes \cO_T \lra H^2(\cO_X)\otimes \cO_T\lra 0\eeq where the differentials are $\lambda\mapsto \lambda\wedge \sum_k t_k\phi_k$. The following is a special case of \cite[Theorem 3.2]{GL2} for the dual $R^2\pi_*(\cL^{-1}\otimes \rho^*K_X)$ of $\pi_*\cL$.
\begin{lemm}\label{thref-GL2} Under the above assumptions, we have isomorphisms
$$(R^i\pi_*(\cL^{-1}\otimes \rho^*K_X))_0\cong \cH^i(D^\bullet)_0$$ where the subscript $0$ means stalk at zero.
\end{lemm}

In fact, the proof of \cite[Theorem 3.2]{GL2}  shows the complex 
 (\ref{eq-GL2}) and the derived object $R^\cdot\pi\lsta(\cL^{-1}\otimes\rho^\ast K_X)$ are quasi-isomorphic over the stalk at $0$. 

As a corollary of \cite[Theorem 3.2]{GL2}, Green and Lazarsfeld also show that \emph{all higher obstructions to deforming a cohomology class vanish automatically along straight lines} (\cite[Corollary 3.3]{GL2}). Hence the first order obstructions define the tangent cone of $M$ at a canonical divisor. 

Let $p=p_g=\dim H^0(K_X)=\dim H^2(\cO_X)>0$. Let us choose a basis $\{\psi_i\}$ of $H^2(\cO_X)\cong \CC^{p}$ and write $$\phi_j\wedge \phi_k=\sum_i a_{ijk}\psi_i$$ with $a_{ijk}=-a_{ikj}\in \CC$.
Then the complex \eqref{eq-GL2} is
\beq\label{eq-GL3} 0\lra \cO_T\mapright{A_0} \cO_T^q\mapright{A_1} \cO_T^p\lra 0\eeq
where $A_0$ is the transpose $(t_1,\cdots, t_n)$ and $$A_1=(\sum_{k}a_{ijk}t_k)_{1\le i\le p, 1\le j\le q}.$$ Upon dualizing, we find that the stalk of $\pi_*\cL$ at $0$ is given by the zeroth cohomology of the complex
\beq\label{eq-GL4} 0\lra \cO^p_T\mapright{A_1^t} \cO_T^q\mapright{A_0^t} \cO_T\lra 0\eeq
where $A_0^t$ and $A_1^t$ are the transpose matrices of $A_0$ and $A_1$ respectively. 

Obviously, a section $s\in H^0(L)$ can be thought of as a closed subscheme of the total space of the line bundle $L$ over $X$. Let $\cU\to P\times X$ be the universal line bundle and consider its composition $\cU\to P\times X\to P$ with the projection to $P$. By the relative Hilbert scheme construction for  $\cU\to P$, we have the moduli scheme $\tilde{M}$ of pairs $(L\in P, s\in H^0(L))$.  Multiplication on $s$ gives us a free action of $\CC^*$ so that $$M=\tilde{M}-P/\CC^*$$ where $P$ is identified with the locus $\{(L\in P, 0\in H^0(L))\}$ of zero sections. 

Fix a basis $\psi_1,\cdots, 
\psi_p$ for $H^0(K_X)$ and let $z_1,\cdots,z_p$ be coordinates of $H^0(K_X)
\cong \CC^p=\{(z_1,\cdots,z_p)\}.$
By \eqref{eq-GL4}, $\tilde{M}$ in an analytic neighborhood of $H^0(K_X)=\{(K_X,s\in H^0(K_X))\}$ is
\beq\label{eq-defeqgam} \Gamma=\{(z_1,\cdots,z_p,t_1,\cdots,t_q)\,|\, \sum_{i,k}z_ia_{ijk}t_k=0\text{ for } 1\le j\le q\}\sub \CC^p\times T\subset \CC^{p}\times\CC^{q}.\eeq

Let $$\xi:\Gamma\hookrightarrow \CC^p\times\CC^q\to \CC^p$$ be given by the projection to the $z$-coordinates. By \eqref{eq-defeqgam}, for a fixed nonzero $z=(z_1,\cdots, z_p)\in \CC^p$, the fiber over $z$ is 
\[ \xi^{-1}(z)=\mathrm{ker} B\subset \CC^q\text{ where } B_{jk}=\sum_iz_ia_{ijk}.\] 
Since $a_{ijk}=-a_{ikj}$, we find that the $q\times q$ matrix $B$ is skew-symmetric. Therefore the rank of $B$ is even and hence $\dim \xi^{-1}(z)\le q$ has the same parity as $q$. 

By upper semi-continuity, there is an open subset $V$ of $\CC^p-\{0\}$ such that, for arbitrary $z\in V$, $\dim \xi^{-1}(z)$ is minimal possible and $\xi^{-1}(V)$ is an open subscheme of a vector bundle over $V$ whose rank has the same parity as $q$. In particular, $\tilde{M}$ is smooth near $V$. Finally we choose a point $Z$ in $(\CC^p\times 0)\cap \xi^{-1}(V)$. Then $Z$ gives us a canonical divisor in $X$ whose corresponding point in $M$ is smooth (after taking quotient by the free $\CC^*$ action). The dimension of $M$ at $Z$ has the parity of $p+q-1$ which equals the parity of $\chi(\cO_X)=p-q+1$. In summary we proved the following.

\begin{prop}\label{tp} 
Let $X$ be a smooth minimal projective surface of general type with $p_g>0$. Then there is a canonical divisor $Z$ in $X$ which represents a smooth point in the Hilbert scheme $Hilb^{k_X}_X$. Furthermore, the dimension of $Hilb^{k_X}_X$ at $Z$ has the same parity as $\chi(\cO_X)$. 
\end{prop}

By the cosection localization in \S\ref{sec3}, the virtual cycle of $Hilb^{k_X}_X$ is localized to any open neighborhood of $Z$. Since $Z$ is a smooth point, we can excise the singular part of $Hilb^{k_X}_X$ and may assume that $Hilb^{k_X}_X$ is a smooth variety whose dimension has the same parity as $\chi(\cO_X)$. Since the virtual dimension is zero, the obstruction sheaf is a locally free sheaf $E$ whose rank equals the dimension of $Hilb^{k_X}_X$. Since $M=Hilb^{k_X}_X$ is smooth, $0:\mathbb{T}_M\to E$ is a perfect obstruction theory. Now Theorem \ref{prref1} follows immediately from Example \ref{exref1} and Lemma \ref{DegIsRed}. This completes our proof of Theorem \ref{thref1}.


\bibliographystyle{amsplain}

\end{document}